\documentclass{amsart}

\usepackage[T5]{fontenc}
\usepackage[utf8]{inputenc}
\usepackage[cyr]{aeguill}
\usepackage[english]{babel}

\usepackage{hyperref}
\usepackage{color}
\usepackage{cite}
\usepackage{enumerate}
\usepackage{hyperref}
\usepackage{geometry}
\usepackage[toc,page]{appendix}

\usepackage{amsmath}
\usepackage{amsfonts}
\usepackage{dsfont}
\usepackage{amssymb}
\usepackage{mathrsfs}
\usepackage[all]{xy}
\usepackage{amsthm}

\DeclareMathOperator{\St}{St}

\DeclareMathOperator{\Ind}{Ind}
\DeclareMathOperator{\PInd}{PInd}

\DeclareMathOperator{\Hom}{Hom}
\DeclareMathOperator{\UU}{\mathcal{U}}
\DeclareMathOperator{\FF}{\mathbb{F}}
\DeclareMathOperator{\CC}{\mathbb{C}}
\DeclareMathOperator{\ZZ}{\mathbb{Z}}
\DeclareMathOperator{\proj}{-proj}
\DeclareMathOperator{\-mod}{-mod}
\DeclareMathOperator{\Stab}{Stab}
\DeclareMathOperator{\Orb}{Orb}

\theoremstyle{definition}
\newtheorem{definition}{Definition}
\newtheorem{exemple}{Example}
\theoremstyle{plain} 
\newtheorem{theoreme}{Theorem}

\newtheorem{lemme}[theoreme]{Lemme}
\newtheorem{prop}[theoreme]{Proposition}
\newtheorem{remarque}[theoreme]{Remark}
\newtheorem{corollaire}[theoreme]{Corollary}
\newtheorem{conjecture}{Conjecture}
\newtheorem*{conjint}{Conjecture}
\newtheorem*{thmint}{Theorem}

\author{Hélène Pérennou}
\address{Universit\'e de Nantes \\
Laboratoire de Math\'ematiques Jean Leray
(UMR 6629 CNRS \& UN)}
\email{helene.perennou@univ-nantes.fr}
\title[General linear groups, Deligne-Lusztig characters and Poincaré series]{Polynomiality of Grothendieck groups for finite general linear groups, Deligne-Lusztig characters, and injective unstable modules}
\begin{document}
\date{\today}
\begin{abstract}
Let $K_0(\FF_p GL_n(\FF_p)\proj)$ denote the Grothendieck group of finitely generated projective $\FF_pGL_n(\FF_p)$-modules. We show that the algebra $\CC \otimes \big(\bigoplus_{n\geq 0}K_0(\FF_p GL_n(\FF_p)\proj)\big)$ with multiplication given by induction functors, is a polynomial algebra. We explicit generators and their relation with Deligne-Lusztig characters.

This work is motivated by several conjectures about unstable modules. Let $K(\UU)$ denote the  Grothendieck group of reduced injective unstable modules of finite type over the mod.\ $p$ Steenrod algebra. We obtain that $\CC\otimes K(\UU)$, with multiplication given by tensor product of unstable modules, is a polynomial algebra. And we identify the Poincaré series associated to elements of $\CC\otimes K(\UU)$.
\end{abstract}
\maketitle
\section*{Introduction}
This paper studies the structure of modular representations of finite general linear groups in natural characteristic. Let $K_0(\FF_p GL_n(\FF_p)\proj)$ denote the Grothendieck group of finite type projective $\FF_p GL_n(\FF_p)$-modules, where $\FF_p$ is the prime field with $p$ elements. The graded group $\bigoplus_{n\geq 0}K_0(\FF_p GL_n(\FF_p)\proj)$ is also a graded ring with multiplication given by induction functors. Using the characterization of modular projective modules by their Brauer characters, we identify the complexification $\CC \otimes\big(\oplus_{n\geq 0} K_0(\FF_p GL_n(\FF_p)\proj)\big)$ with the set of class functions over $GL_n(\FF_p)$ vanishing outside semi-simple classes. A basis of $\CC \otimes \big(\bigoplus_{n\geq 0}K_0(\FF_p GL_n(\FF_p)\proj)\big)$ is given by the characteristic functions of semi-simple classes. A semi-simple class is entirely determined by the characteristic polynomial. Thus, monic polynomials over $\FF_p$ with a non-zero constant coefficient parametrize the characteristic functions of semi-simple classes. It happens that this indexing behaves well under multiplication. This gives the following,
\begin{thmint}\label{thmintro}
The algebra $\CC \otimes \big(\bigoplus_{n\geq 0}K_0(\FF_p GL_n(\FF_p)\proj)\big)$ is polynomial and a family of generators is given by the characteristic functions associated to irreducible polynomials over $\FF_p$ with a non-zero constant coefficient.
\end{thmint}
Tensoring Steinberg modules $\St_n$ with Deligne-Lusztig \cite{DL} characters provides actual $\FF_p GL_n(\FF_p)$-projective modules. We use it to give another description of the polynomial generators.

These results on modular characters of general linear groups are motivated by conjectures coming from the study of modules over the Steenrod algebra in algebraic topology. Denote by $K^n(\UU)$ the free abelian group generated by isomorphism classes of direct summands in the Steenrod algebra module $H^*V_n:=H^*(V_n,\ZZ/p)$, where $V_n=(\FF_p)^n$. Viewing a summand of $H^*V_n$ as a summand of $H^*V_{n+1}$ induces a monomorphism from $K^n(\UU)$ to $K^{n+1}(\UU)$. Thus $K(\UU)=\bigcup_{n\geq 0} K^n(\UU)$ is a filtered free abelian group. It is also a ring with multiplication given by the tensor product of unstable modules. Carlisle and Kuhn state the following conjecture about this structure.
\begin{conjint}\cite[\S 4]{CarlisleKuhnSma}\label{conjpoly}
The ring $K(\UU)$ is polynomial.
\end{conjint}
It is known by \cite[Thm. 3.2]{CarlisleKuhnSma} that the rings $\bigoplus_{n\geq 0}K_0(\FF_p GL_n(\FF_p)\proj)$ and $K(\UU)$ are isomorphic. Theorem \ref{thmintro} then shows that $K_{\CC}(\UU):=\CC\otimes K(\UU)$ is a polynomial algebra. Indeed the Conjecture \ref{conjpoly} holds also over the rationals. We describe a family of polynomial generators (over $\CC$) using Campbell-Selick direct summands \cite{CS}, as in \cite{Hai19}.
This gives a new proof of a conjecture of Schwartz \cite[Conj. 3.2]{DHS}, already proved by Hai \cite{HaiTop,HaiAlg,Hai19}, see also \cite{FHS}.
\begin{prop}
The action of Lannes' $T$ functor on $K_{\CC}^n(\UU)$ is diagonalizable with spectrum $\{1,p,\ldots, p^n\}$.
\end{prop}
We also identify Poincaré series associated to elements of $K_{\CC}(\UU)$. Let $P_{\CC}(.,t)$ denotes the algebra map from $K_{\CC}(\UU)$ to $\CC[[t]]$ which associates to the class of a module $M$, its Poincaré series. Fix $\theta$ an embedding of $\overline{\FF}_p^{\times}$ in $\CC$.
\begin{prop}
The image of $P_{\CC}(.,t)$ is generated as an algebra by the following series
$$P_{\mathbb{C}}(P_f,t)= 
\left\{\begin{array}{ccc}
\frac{1}{(w-t)(w^{2}-t)\ldots(w^{2^{n-1}}-t)}, & & \text{ for } p=2,\\
& & \\
\frac{(w+t)(w^{p}+t)\ldots(w^{p^{n-1}}+t)}{(w-t^2)(w^{p}-t^2)\ldots(w^{p^{n-1}}-t^2)}, & & \text{ for } p>2,
\end{array}\right.$$
where $f$ runs through the set of irreducible polynomials over $\FF_p$ with a non-zero constant term and $w=\theta(\alpha)$ with $\alpha$ a root of $f$.
\end{prop}

We recover, that the kernel of $P_{\CC}(.,t)$ is non-trivial for $p=2$ \cite{DHS} and our argument answers a more precise question by Hai \cite[\S 5]{Hai19} :
\begin{prop}
For $p=2$, the restriction of $P_{\CC}(.,t)$ to the $1$-eigenspace of $K_{\CC}(\UU)$ admits a non-trivial kernel.
\end{prop}

\subsection*{Organisation of the paper}
The first section is devoted to the description of the polynomial structure of $\mathbb{C}\otimes \big( \bigoplus_{n\geq 0}K_0(\FF_p GL_n(\FF_p)\proj)\big)$. Section 2 explains the link with Deligne-Lusztig characters. Finally, section 3 gives applications to unstable modules over the Steenrod algebra.
\section{Polynomiality}
Let $A$ be a ring, and denote by $K_0(A\-mod)$ (resp. $K_0(A\proj)$) the Grothendieck group of $A$-modules (resp. projective $A$-modules) of finite type. The graded groups $\bigoplus_{n\geq 0}K_0(\FF_p GL_n(\FF_p)\-mod)$ and $\bigoplus_{n\geq 0}K_0(\FF_p GL_n(\FF_p)\proj)$ are also commutative rings with multiplication given by the induction functors $\Ind_{GL_k(\FF_p\times GL_l(\FF_p)}^{GL_n(\FF_p)}(\_)$ where $k$, $l$ et $n$ are natural numbers satisfying $k+l=n$. Given a (projective) finitely generated $\FF_p GL_n(\FF_p)$-module $V$, denote by $[V]$ its class in $K_0(\FF_p GL_n(\FF_p)\-mod)$ (or in $K_0(\FF_p GL_n(\FF_p)\proj)$). In this first section, we show that $\CC\otimes \bigoplus_{n\geq 0}K_0(\FF_p GL_n(\FF_p)\proj)$ is a polynomial algebra by giving an explicit family of generators. 
\begin{remarque}
We work over $\FF_p$. Since it is a splitting field for $GL_n(\FF_p)$, one can replace it by any $\FF_q$, where $q$ is a power of $p$.
\end{remarque}
\subsection{Characters of general linear groups}
Using the link between a projective module and its Brauer character, for each natural number $n$, we identify $\mathbb{C}\otimes K_0(\FF_p GL_n(\FF_p)\proj)$ with the space of class functions over $GL_n(\FF_p)$, vanishing outside $p$-regular elements (\emph{i.e} elements of order prime to $p$). Recall that an element of $GL_n(\FF_p)$ has order prime to $p$ if, and only if, it is semi-simple (\emph{i.e.} diagonalizable in a finite extension of $\FF_p$). With this point of view, the multiplication in $\CC\otimes K_0(\FF_p GL_n(\FF_p)\proj)$ is given by the formula :
$$\rho_{n}.\rho_{m}(g) = \displaystyle \frac{1}{|GL_n(\mathbb{F}_p)|.|GL_m(\mathbb{F}_p)|} \sum_{ \substack{ h\in GL_{n+m}(\FF_p),\\  hgh^{-1}=\tiny\begin{pmatrix}
g_n & 0 \\ 0 & g_m
\end{pmatrix}}} \rho_{n}(g_n)\rho_{m}(g_m),$$
where $\rho_n$ and $\rho_m$ are $p$-regular class functions over $GL_n(\FF_p)$ and $GL_m(\FF_p)$ respectively.

For $g$ in $GL_n(\FF_p)$, denote by $\chi_g$ the characteristic polynomial of $g$.
\begin{definition}
Let $f$ in $\FF_p[x]$ of degree $n$, $f\neq x$. We denote by $\pi_f$ the class function over $GL_n(\FF_p)$ defined by
$$\pi_f(g)= \delta_{\chi_g,f},$$
where $\delta$ is the Kronecker's symbol.
\end{definition}
Thus, $\pi_f$ is the characteristic function of the conjugacy class of semi-simple elements of $GL_n(\FF_p)$ for which the characteristic polynomial is $f$. This choice of index coincides with the multiplication in the following way. For all $f_n$, $f_m$ in $\FF_p[x]$ of degree $n$ and $m$ respectively, and $g$ in $GL_{n+m}(\FF_p)$,
\begin{equation}\label{prodcar}
\pi_{f_n}.\pi_{f_m}(g) = \displaystyle \frac{1}{|GL_n(\mathbb{F}_p)|.|GL_m(\mathbb{F}_p)|} \sum_{ \substack{ h\in GL_{n+m}(\mathbb{F}_p),\\  hgh^{-1}=\tiny\begin{pmatrix} g_n & 0 \\ 0 & g_m\end{pmatrix}}} \pi_{f_n}(g_n)\pi_{f_m}(g_m) = c_{f_n,f_m}\pi_{f_n.f_m}
\end{equation}
where $c_{f_n,f_m}$ is an integer.
Thus, we have the following.
\begin{theoreme}\label{poly}
The algebra $\mathbb{C}\otimes \big( \bigoplus_{n\geq 0}K_0(\FF_p GL_n(\FF_p)\proj)\big)$ is polynomial with one generator for each irreducible polynomial in $\FF_p [x]$ with non-zero constant coefficient. Precisely,
$$\mathbb{C}\otimes \big( \bigoplus_{n\geq 0}K_0(\FF_p GL_n(\FF_p)\proj)\big) \cong \CC[\pi_f, f\text{ irreducible in } \FF_p [x], f\neq x].$$
\end{theoreme}

\begin{corollaire}\label{corQ}
The algebra $\mathbb{Q}\otimes \big( \bigoplus_{n\geq 0}K_0(\FF_p GL_n(\FF_p)\proj)\big) $ is polynomial with one generator for each irreducible polynomial in $\FF_p [x]$ with non-zero constant coefficient.
\end{corollaire}

\subsection{Multiplicative constants}\label{sectionmultcont}
In order to complete the description of this algebra, it remains to compute the constants $c_{f_n,f_m}$ from (\ref{prodcar}).

Consider the ring $\bigoplus_{n\geq 0} K_0(\CC GL_n(\FF_p)\-mod)$ where now the multiplication is defined using parabolic induction $\PInd$ (also called Harish-Chandra induction). The functor 
$$\PInd : \CC GL_k(\FF_p)\-mod\times \CC GL_l(\FF_p)\-mod \rightarrow \CC GL_n(\FF_p)\-mod,$$ with $k+l=n$, is the inflation from $GL_k(\FF_p)\times GL_l(\FF_p)$ to the parabolic subgroup $P_{k,l}$ of $GL_k(\FF_p)$, composed with the (ordinary) induction $\Ind_{P_{k,l}}^{GL_n(\FF_p)}(\_)$ (see \cite[\S 8]{Zel}).
This ring is studied by Springer and Zelevinsky in \cite{Zel,SpringZel}. They focus on its polynomial structure. We also define the graded ring $(\bigoplus_{n\geq 0} K_0(\FF_p GL_n(\FF_p)\-mod), \PInd )$ in a same way.

The main tools to link these rings are the Steinberg modules, denoted by $\St_n$, and Brauer theory.

\begin{theoreme}\cite{Lusztig}\cite[9.6]{DigneMichel}\label{Stiso}
The map 
$$\begin{array}{ccc}
(\bigoplus_{n\geq 0} K_0(\FF_p GL_n(\FF_p)\-mod), \PInd ) & \rightarrow & (\bigoplus_{n\geq 0} K_0(\FF_p GL_n(\FF_p)\proj ), \Ind) \\
\left[V\right] & \mapsto & [\St_n\otimes V]
\end{array}$$
is a ring isomorphism.
\end{theoreme}

\begin{corollaire}
The map 
$$\begin{array}{ccc}
(\bigoplus_{n\geq 0} K_0(\CC GL_n(\FF_p)\-mod), \PInd ) & \rightarrow & (\bigoplus_{n\geq 0} K_0(\FF_p GL_n(\FF_p)\proj ), \Ind) \\
\left[V\right] & \mapsto & [\St_n\otimes d_n(V)]
\end{array}$$
where $d_n : \bigoplus_{n\geq 0} K_0(\CC GL_n(\FF_p)\-mod)\mapsto \bigoplus_{n\geq 0} K_0(\FF_p GL_n(\FF_p)\-mod)$ is the reduction map \cite[\S 15.2]{Serre}, is an epimorphism.
\end{corollaire}

We can now compute the constants from (\ref{prodcar}).
\begin{prop}\label{psi}
Let $f$ and $g$ be monic polynomials in $\FF_p[x]$ with non-zero constant terms.
\begin{itemize}
\item If $f$ and $g$ are relatively prime, $\pi_f.\pi_g = \pi_{f.g}$, that is
$$c_{f,g} = 1.$$
\item If $f$ is irreducible of degree $d$. For all natural numbers $n$ and $m$,
$$c_{f^n,f^m} = p^{dmn}\frac{\psi_{n+m}(p^d)}{\psi_n(p^d)\psi_m(p^d)}$$
where $\psi_k(p^d)=\prod_{i=1}^k(p^{id}-1)$.
\end{itemize}
\end{prop}
\begin{proof}
The first point follows from \cite[\S 1.4]{SpringZel}. For the second point let $n$ be any integer. The Steinberg character is described as follow. Let $f_1,\ldots, f_k$ be irreducible polynomials over $\FF_p$ with non-zero constant term and of degree $d_1,\ldots,d_k$ respectively. For a semi-simple element $g$ in $GL_n(\FF_p)$ with characteristic polynomial $f_1^{n_1}\ldots f_k^{n_k}$,
\begin{equation}\label{carSt}
St_n(g)=(-1)^{N}p^{D},
\end{equation}
where $N= n-\sum_{i=1}^k n_i$ et $D=\sum_{i=1}^{k} d_i\binom{n_i}{2}$ (see \cite[\S 1.13]{SpringZel}).

Thus, $$\pi_{f^n}=\dfrac{(-1)^{nd-n}}{p^{d\binom{n}{2}}}\St_{dn}\otimes\pi_{f^n}
\text{ \ and \ }
\pi_{f^m}=\dfrac{(-1)^{md-m}}{p^{d\binom{m}{2}}}\St_{dm}\otimes\pi_{f^m}.$$ So, 
$$\pi_{f^n}.\pi_{f^m}=\frac{(-1)^{d(m+n)-m+n}}{p^{d(\binom{n}{2}+\binom{m}{2})}}\St_{d(n+m)}\otimes\PInd_{GL_{dn}(\FF_p)\times GL_{dm}(\FF_p)}^{GL_{d(n+m)}(\FF_p)}(\pi_{f^n}\otimes\pi_{f^m}).$$
And by Theorem \ref{Stiso} and \cite[Prop. 10.1]{Zel}, 
$$\St_{d(n+m)}\otimes\PInd_{GL_{dn}(\FF_p)\times GL_{dm}(\FF_p)}^{GL_{d(n+m)}(\FF_p)}(\pi_{f^n}\otimes\pi_{f^m})=\St_{d(n+m)}\otimes\frac{\psi_{n+m}(p^d)}{\psi_n(p^d)\psi_m(p^d)}\pi_{f^{n+m}}.$$
The result follows again by (\ref{carSt}).
\end{proof}

\begin{corollaire}\label{puissance}
Let $f$ be an irreducible polynomial of degree $d$ over $\FF_p$ with a non-zero constant coefficient,
$$(\pi_f)^n=p^{d\binom{n}{2}}\frac{\psi_n(p^d)}{\psi_1(p^d)^n}\pi_{f^n}.$$
\end{corollaire}
\subsection{Grouplike elements}
The ring $\bigoplus_{n\geq 0} K_0(GL_n(\FF_p))$ is also a bi-algebra \cite[Thm. 6.10]{Kuhn2} (which is not graded) with co-multiplication $\Delta$ defined in degree $n$ by 
$$
\Ind_{GL_n(\FF_p)}^{GL_n(\FF_p)\times GL_n(\FF_p)} : K_0(\FF_p GL_n(\mathbb{F}_p)\proj) \rightarrow  K_0(\FF_p(GL_n(\mathbb{F}_p)\times GL_n(\mathbb{F}_p))\proj) $$
composed with the inverse of
$$\_\otimes\_ :  K_0(\FF_p GL_n(\mathbb{F}_p)\proj)\otimes K_0(\FF_p GL_n(\mathbb{F}_p)\proj)\rightarrow K_0(\FF_p(GL_n(\mathbb{F}_p)\times GL_n(\mathbb{F}_p))\proj)
.$$
The functions $\pi_f$ have the following nice property.
\begin{lemme}
For $f$ a monic polynomial over $\FF_p$ with a non-zero constant coefficient, $\pi_f$ is grouplike:
$$\Delta \pi_f=\pi_f \otimes \pi_f.$$
\end{lemme}
\begin{proof}
Let $f$ be as above and of degree $n$. Fix $(a,b)$ in $GL_n(\FF_p)\times GL_n(\FF_p)$.
The induction formula for class functions gives :
$$
\begin{array}{ccc}
\Ind_{GL_n(\FF_p)}^{GL_n(\FF_p)\times GL_n(\FF_p)}(\pi_f)(a,b) & = & \frac{1}{|GL_n(\FF_p)|}\displaystyle{\sum_{\substack{(h,l)\in GL_n(\FF_p)\times GL_n(\FF_p) \\ (hah^{-1}, lbl^{-1})\in GL_n(\FF_p)}} \pi_f(hah^{-1})}\\
& = & \frac{1}{|GL_n(\FF_p)|}\displaystyle{\sum_{\substack{(h,l)\in GL_n(\FF_p)\times GL_n(\FF_p) \\ hah^{-1}= lbl^{-1}}} \pi_f(hah^{-1})}
\end{array}$$
Thus
$$ \begin{array}{ccc}
\Ind_{GL_n(\FF_p)}^{GL_n(\FF_p)\times GL_n(\FF_p)}(\pi_f)(a,b)
&= & \Ind_{GL_n(\FF_p)}^{GL_n(\FF_p)\times GL_n(\FF_p)}(\pi_f)(a,a)\delta_{\chi_a,\chi_b} \\
&  & \\
& = & \frac{1}{|GL_n(\FF_p)|}\displaystyle{\sum_{\substack{(h,l)\in GL_n(\FF_p)\times GL_n(\FF_p) \\ hah^{-1}= lal^{-1}}} \pi_f(a)\delta_{\chi_a,\chi_b}}\\
\end{array}.$$
The one to one correspondence between $$\{(h,l)\in GL_n(\FF_p)\times GL_n(\FF_p) |  hah^{-1}=lal^{-1} \} \text{ and }\Stab_{GL_n(\FF_p)}(a)\times \Orb(a),$$
gives $$\#\{(h,l)\in GL_n(\FF_p)\times GL_n(\FF_p) |  hah^{-1}=lal^{-1} \}=|GL_n(\FF_p)|$$ and thus,
$$\Ind_{GL_n(\FF_p)}^{GL_n(\FF_p)\times GL_n(\FF_p)}(\pi_f)(a,a)=\pi_f(a)=\delta_{\chi_a,f}.$$
Finally,
$$\Ind_{GL_n(\FF_p)}^{GL_n(\FF_p)\times GL_n(\FF_p)}(\pi_f)(a,b) =  \delta_{(\chi_a,\chi_b),(f,f)}$$

Thus $\Ind_{GL_n(\FF_p)}^{GL_n(\FF_p)\times GL_n(\FF_p)}(\pi_f)$ is the image of $\pi_f\otimes \pi_f$ by $$K_0(\FF_p GL_n(\mathbb{F}_p)\proj)\otimes K_0(\FF_p GL_n(\mathbb{F}_p)\proj) \overset{\otimes}{\longrightarrow}K_0(\FF_p (GL_n(\mathbb{F}_p)\times GL_n(\mathbb{F}_p))\proj)$$ and $\Delta(\pi_f)=\pi_f\otimes\pi_f$.
\end{proof}
\section{Deligne-Lusztig characters}
In this section, we use Deligne-Lusztig characters to give another description of the class functions $\pi_f$. Let $\alpha$ be a primitive root of degree $n$ over $\FF_p$ (that is a cyclic generator of $\FF_{p^n}^{\times}$), let $\theta$ be an embedding of $\overline{\mathbb{F}}_p^{\times}$ in $\mathbb{C}^{\times}$ and let $w=\theta(\alpha)$. We denoted by $T_n$ the cyclic group generated by $\alpha$, thus $T_n\cong\mathbb{Z}/(p^n-1)$. The group $T_n$ can be view as a subgroup of $GL_n(\mathbb{F}_p)$ by choosing a basis of $\mathbb{F}_{p^n}$ as an $\FF_p$-vector space. Finally, denote for all natural number $i$ in $\{0,\ldots, p^n-2\}$,
$$\begin{array}{ccccc}
\varphi_i & : & T_n & \rightarrow & \mathbb{C}^{\times} \\
 &  & \alpha^k & \mapsto & \omega^{ki}
\end{array}$$
the irreducible characters of $T_n$. The induced characters $\Ind_{T_n}^{GL_n(\FF_p)}(\varphi_i)$, for $i$ in $\{0,\ldots, p^n-2\}$, are projective and they only depend on the orbit of $\alpha^i$ under the action of the Frobenius map. Moreover, they are closely related to Deligne-Lusztig characters.

\begin{lemme}\cite[Prop. 7.3]{DL}\label{indDL}
For all $i$ in $\{0,\ldots, p^n-2\}$,
$$\Ind_{T_n}^{GL_n(\FF_p)}(\varphi_i)=(-1)^{n-1} R_{T_n}^{\varphi_i}\otimes \St_n,$$
where $R_{T_n}^{\varphi_i}$ is the Deligne-Lusztig character associated to the character $\varphi_i$ of $T_n$  \cite{DL}.
\end{lemme}

Now consider, for all $i$ in $\{0,\ldots, p^n-2\}$, $\hat{\varphi_i}$ the Fourier transform of $\varphi_i$. For all $i$ and $k$ in $\{0,\ldots p^n-2\}$,
$$\hat{\varphi_i}(\alpha^k)  =  \frac{1}{p^n-1}\sum_{j=0}^{p^n-2}w^{-kj}\varphi_{i}(\alpha^j) =\delta_{i,k}.$$
Thus, $\hat{\varphi_i}$ is the indicator function of $\alpha^i$ over $T_n$. We have the following identity,

\begin{lemme}\label{varphi1}
For all $k$ in $\{0,\ldots,p^n-2\}$,
$$\hat{\varphi}_k = \frac{1}{p^n-1} \sum_{j=0}^{p^n-2} w^{-jk}\varphi_j.$$
\end{lemme}
\noindent
Lemma \ref{varphi1} gives a particular case of \cite[Prop. 7.5]{DL} :
\begin{equation}\label{hatvarphi}
\Ind_{T_n}^{GL_n(\FF_p)}(\hat{\varphi}_k)=\frac{(-1)^{n-1}}{p^n-1} \sum_{j=0}^{p^n-2} w^{-jk} R_{T_n}^{\varphi_j}\otimes \St_n.
\end{equation}

\begin{prop}\label{pifDL}
For all $k$ in $\{0,\ldots, p^n-2\}$,
$$\Ind_{T_n}^{GL_n(\FF_p)}(\hat{\varphi}_k)=\frac{\psi_{m_k}(p^{d_k})}{\psi_1(p^n)}\pi_{f_{k}^{m_k}}$$
where $f_k$ is the irreducible polynomial over $\FF_p$ with roots $\{\alpha^k,\alpha^{pk},\ldots,\alpha^{p^{n-1}k}\}$, $d_k$ is the degree of $f_k$, $m_k=n/d_k$, and $\psi_m(p^d)=\prod_{i=1}^m(p^{id}-1)$ as in Proposition \ref{psi}.
\end{prop}
\begin{proof}
We saw that $\hat{\varphi}_k$ is the indicator function of $\alpha^k$. Thus $\Ind_{T_n}^{GL_n(\FF_p)}(\hat{\varphi}_k)$ is collinear to the characteristic function of the conjugacy class of $\alpha^k$ in $GL_n(\FF_p)$. This is exactly the class \hbox{function $\pi_{f_{k}^{m_k}}$.}
It remains to compute $\Ind_{T_n}^{GL_n(\FF_p)}(\hat{\varphi}_k)(\alpha^k)$ :
$$\Ind_{T_n}^{GL_n}(\hat{\varphi_k})(\alpha^k)=\frac{1}{|T_n|}\sum_{ \substack{ h\in GL_{n}(\mathbb{F}_p),\\  h\alpha^k h^{-1} \in T_n}}\hat{\varphi_k}(\alpha^k) = \frac{|Z_{GL_n(\FF_p)}(\alpha^k))|}{|T_n|}$$
Now $|Z_{GL_n(\FF_p)}(\alpha^k))|$ depends on the orbit of the $\alpha^k$ under the Frobenius. Denote by $d_k$ the cardinal of  $\{\alpha^k,\alpha^{pk},\ldots, \alpha^{kp^{n-1}})$, the orbit of $\alpha^k$, and $m_k=n/d_k$. One has,
$$|Z_{GL_n(\FF_p)}(\alpha^k))| =  |GL_{m_k}(\FF_{p^{d_k}})|.$$
Thus,
$$
\begin{array}{ccc}
\frac{|Z_{GL_n(\FF_p)}(\alpha^k))|}{|T_n|}
 & = & \frac{(p^{d_km_k}-1)(p^{d_km_k}-p^{d_k})\ldots(p^{d_km_k}-p^{d_km_k-d_k})}{p^n-1}\\
 & & \\
 & = & \frac{\psi_{m_k}(p^{d_k})}{\psi_1(p^n)}.
 \end{array}$$
\end{proof}
Recall from Corollary \ref{puissance} that for $f_k$ irreducible monic polynomial with roots set $\{\alpha^k,\alpha^{pk},\ldots,\alpha^{p^{n-1}k}\}$, 
$$\pi_{f_k}^{m_k} = p^{d_k\binom{m_k}{2}}\frac{\psi_{m_k}(p^{d_k})}{\psi_1(p^{d_k})^{m_k}}\pi_{f_k^{m_k}}.$$
Then Proposition \ref{pifDL} implies :
\begin{prop}\label{DLpif}
$$\pi_{f_k}^{m_k} = (-1)^{n-1} \frac{p^{d_k\binom{m_k}{2}}}{(p^{d_k}-1)^{m_k}}\sum_{j=0}^{p^n-2} w^{-jk} R_{T_n}^{\varphi_j}\otimes \St_n,$$
and inverting the Fourier transform gives,
$$R_{T_n}^{\varphi_i}\otimes \St_n =\sum_{k=0}^{p^n-2}w^{ik} \frac{\psi_1(p^{d_k})^{m_k}}{p^{d_k\binom{m_k}{2}}\psi_1(p^n)}\pi_{f_{k}}^{m_k}.$$
\end{prop}
Thus, we can choose polynomial generators for $\mathbb{C}\otimes \big( \bigoplus_{n\geq 0}K_0(\FF_p GL_n(\FF_p)\proj)\big)$ among the representations $(-1)^{n-1}R_{T_n}^{\varphi_j}\otimes \St_n$.
\section{Applications for unstable modules}
We use the notations of the introduction. 
A conjecture of Carlisle et Kuhn \cite[\S 8]{CarlisleKuhnSma} states that,
\begin{conjecture}\cite[\S 4]{CarlisleKuhnSma}\label{conjCK}
The ring $K(\UU)$ is polynomial over $\mathbb{Z}$.
\end{conjecture}

This conjecture is discussed in\cite[\S 4]{Hai19}. This conjecture translates in terms of modular representations of the general linear groups as follows. For a projective $\FF_p GL_n(\FF_p)$-module $P$, consider the unstable module $\Hom_{GL_n(\FF_p)}(P,H^*V_n)$. It is isomorphic to a direct summand of $H^*V_n$.
The resulting map
\begin{equation}\label{iso}
\begin{array}{ccc}
\bigoplus_{n\geq 0} K_0(\FF_p GL_n(\FF_p)\proj) & \rightarrow & K(\mathcal{U}) \\
\left[P\right] & \mapsto &  [\Hom_{GL_n(\FF_p)}(P,H^*V_n)],
\end{array}
\end{equation}
is a ring isomorphism by \cite[Thm. 3.2]{CarlisleKuhnSma}. We extend it to an $\CC$-algebra  isomorphism,
\begin{equation}\label{isoC}
\CC\otimes\big(\bigoplus_{n\geq 0}K_0(\FF_p GL_n(\FF_p)\proj)\big) \overset{\sim}{\longrightarrow} K_{\CC}(\UU).
\end{equation}
Theorem \ref{poly} proves a weak form of the statement of Conjecture \ref{conjCK}.
\begin{theoreme}
The algebra $K_{\CC}(\UU)$ is polynomial and a family of generators is 
$$\{P_f,\text{ $f$ in $\FF_p[x]$ irreducible, } f(0)\neq 0\},$$ where $P_f$ is the image of $\pi_f$ by the isomorphism (\ref{isoC}).
\end{theoreme}
\noindent Corollary \ref{corQ} gives the result over the rationals.
\begin{corollaire}
The algebra $\mathbb{Q}\otimes K(\UU)$ is polynomial with one generator for each irreducible polynomial over $\FF_p$ with a non-zero constant coefficient.
\end{corollaire}

Recall from Proposition \ref{pifDL} that the class functions $\pi_f$ are complex linear combination of tensor products of Deligne-Lusztig characters with $\St_n$. The isomorphism (\ref{iso}) provides a one-to-one correspondence between Campbell-Selick \cite{CS} direct summands $M_n(j)$ and characters $\St_n\otimes R_{T_n}^{\varphi_i}$ \cite[\S 4]{Hai19}. Proposition \ref{DLpif} now reads :
$$ P_f = \frac{1}{p^n-1}\sum_{j=0}^{|T_n|-1}w^{-j}M_n(j).$$
Note that summands $M_n(j)$ were used by Hai to describe eigenvectors of the Lannes' $T$ functor over the rationals \cite[\S 1]{Hai19}.
\subsection{Eigenvectors of Lannes' $T$ functor}
We now consider the action of the Lannes' $T$ functor \cite{Lannes92}. The functor $T$ is defined as the left adjoint to $\_\otimes H^*\FF_p : \mathcal{U} \rightarrow \mathcal{U}$. This is an exact functor and it commutes with tensor product \cite{Lannes92}. In particular, it induces a ring endomorphism of $K(\mathcal{U})$, and we still denote it by $T$. By\cite{Lannes92},
$$T(\Hom_{GL_n(\FF_p)}(P,H^*V_n))\cong \Hom_{GL_n(\FF_p)}(P\otimes\mathbb{F}_p[V_n^*],H^*V_n).$$
By the isomorphism (\ref{iso}), we also consider $T$ as an endomorphism of $\bigoplus_{n\geq 0} K_0(GL_n(\FF_p)\proj)$. Then,
$$T([P])=[P\otimes\mathbb{F}_p[V_n^*]].$$
Let $T_{\CC}$ denote the complexification of $T$.
\begin{prop}
Let $f=(x-1)^n g$ in $\FF_p [x]$ with $g(0)g(1)\neq 0$. The function $\pi_f$ is an $p^n$-eigenvector of $T_{\CC}$.
\end{prop}
\begin{proof}
Let $f$ be a polynomial of degree  $n$ over $\FF_p$ with a non-zero constant coefficient. One has, $T(\pi_f)=\pi_f\otimes\mathbb{F}_p[V_n^*]$. Then , let $g$ be in $GL_n(\mathbb{F}_p)$ such that its characteristic polynomial is $f$. One has 
$$\pi_f\otimes\mathbb{F}_p[V_n^*](g) = \left\{\begin{array}{cc}
\mathbb{F}_p[V_n^*](g), & \text{ if } \chi_g=f\\
0 & \text{ else}.
\end{array}\right.$$
If $f\neq x-1$ and $\chi_g=f$, then $g:V_n^*\rightarrow V_n^*$ has no fixed point. Thus $\mathbb{F}_p[V_n^*](g)= 1$ and $T(\pi_f)=\pi_f$. Else, $f=x-1$ and for $g$ satisfying $\chi_g=f$, one has $\mathbb{F}_p[V_1^*](g)= p$. Thus $T(\pi_f)=p\pi_f$.
\end{proof}

As a corollary, we recover Schwartz' conjecture \cite[Conj. 3.2]{DHS} about the diagonalization of $T_{\CC}$. 
This conjecture was proved in \cite{HaiTop} and \cite{HaiAlg} by different methods.

\begin{corollaire}
The action of $T$ on $K_{\CC}^n(\UU)$ is diagonalizable and its eigenvalues are $1,p,\ldots, p^n$. Furthermore, for $0\leq i \leq 1$, a basis of the $p^i$-eigenspace is $$\big((H^*\FF_p)^{\otimes i} \otimes P_f,\ \deg(f)=n-i,\ f(1)\neq 0\big).$$ Its dimension is $p^{n-i}-p^{n-i-1}$ for $i<n$, and it is one-dimensional for $i=n$.
\end{corollaire}
\subsection{Poincaré series}
To conclude, we identify the Poincaré series associated to elements of $K_{\CC}(\UU)$. As in the introduction, let $P(.,t)$ denote the map which associates to the class of an unstable module its Poincaré serie,
$$\begin{array}{ccc}
K(\mathcal{U}) & \rightarrow & \mathbb{Z}[[t]] \\
 \left[M\right] & \mapsto & \sum_{d\geq 0} \dim M^d t^d.
\end{array}$$
We denote by $P_{\CC}(.,t)$ its complexification. Finally, recall that for $f$ a polynomial over $\FF_p$ with a non-zero constant coefficient, $P_f$ denotes the element of $K_{\CC}(\UU)$ associated to $\pi_f$. The following proposition describes the image of $P_{\mathbb{C}}(.,t)$.
\begin{prop}\label{poincare}
Let $f$ be an irreducible polynomial over $\FF_p$ with a non-zero constant term,
$$P_{\mathbb{C}}(P_f,t)= 
\left\{\begin{array}{ccc}
\frac{1}{(w-t)(w^{2}-t)\ldots(w^{2^{n-1}}-t)}, & & p=2,\\
& & \\
\frac{(w+t)(w^{p}+t)\ldots(w^{p^{n-1}}+t)}{(w-t^2)(w^{p}-t^2)\ldots(w^{p^{n-1}}-t^2)}, & & p>2.
\end{array}\right.$$
\end{prop}
\begin{proof}
Since
$$ P_f = \frac{1}{p^n -1}\sum_{j=0}^{|T_n|-1}w^{-j}M_n(j),$$ and
$$M_n(j)=\Hom_{GL_n(\FF_p)}(\Ind_{T_n}^{GL_n(\FF_p)}(\varphi_j),H^*V_n),$$
the result follows from Molien's formula.
\end{proof}

\subsection*{The case $\mathbf{p=2}$}
In this case, the formula in the Proposition \ref{poincare} takes a simple form.

Let $f$ be an irreducible polynomial over $\mathbb{F}_2$ and $\alpha_1,\ldots,\alpha_n$ its roots in $\overline{\mathbb{F}}_2^{\times}$. We denote by $\tilde{f}$ the polynomial $(\theta(\alpha_1)-t)\ldots(\theta(\alpha_n)-t)$ in $\CC[t]$. The image of $P_{\mathbb{C}}(.,t)$ is the subalgebra of $\mathbb{C}(t)$ generated by $\{1/\tilde{f}\  |\  f\in \mathbb{F}_2[x],\ f(0)\neq 0\}$.

Let $g$ be the product of irreducible polynomials $f_1, \ldots, f_N$, we denote $P_g=P_{f_1}\otimes\ldots \otimes P_{f_N}$ and $\tilde{g}=\tilde{f}_1\ldots\tilde{f}_N$. Thus we have $P_{\mathbb{C}}(P_g,t)=1/\tilde{g}$.

\begin{prop}
For $p=2$, the kernel $P_{\mathbb{C}}(.,t)$ is non-trivial.
\end{prop}
\begin{proof}
Let $N$ be a natural number and $f_1,\ldots, f_{k_N}$ be irreducible polynomial over $\mathbb{F}_2$ of degree $N$. By Proposition \ref{poincare}, for all $\alpha_1, \ldots, \alpha_{k_N}$, in $\CC$, the serie
$$\sum_{i=1}^{k_N} \frac{\alpha_i}{\prod_{j=1,j\neq i}^{k_N} \tilde{f}_j} = \sum_{i=1}^{k_N} \frac{\alpha_i \tilde{f}_i}{\prod_{j=1}^{k_N} \tilde{f}_j}$$
is in the image. Thus, every equation
$$\alpha_1\tilde{f}_1+\ldots +\alpha_N\tilde{f}_N=0$$
in $\mathbb{C}[t]$ gives an element in the kernel. In particular, when the number of irreducible polynomials of degree $N$ is strictly bigger that the dimension of the vector space $\CC[t]_{\leq n}$, there is this kind of relation. For example, it happens in degree $N=6$. 
\end{proof}
In particular, the proof gives a negative answer to the question raised in\cite[4.6]{Hai19} : the restriction of $P_{\mathbb{C}}(.,t)$ on the $1$-eigenspace of $K_{\CC}(\UU)$ for $T$ is not injective. 

We can also describe the kernel of $P_{\mathbb{C}}(.,t)$. Every element of $P$ of $K_{\CC}(\UU)$ is a linear combination of the $P_{f_1},\ldots,P_{f_N}$, where $f_i$ is monic of degree $d_i$ with a non-zero constant coefficient :
$$ P= \sum_{i=1}^N \alpha_{f_i} P_{f_i}.$$
So $P$ is in the kernel if, and only if,
$$\sum_{i=1}^N \frac{\alpha_{f_i}}{ \tilde{f}_i}= 0.$$
That is,
$$ \sum_{i=1}^N \frac{\alpha_{f_i}\prod_{j=1,j\neq i}^N \tilde{f}_i}{\prod_{j=1}^N \tilde{f}_i}= 0,$$
or,
$$\sum_{i=1}^N \alpha_{f_i}\prod_{j=1,j\neq i}^N \tilde{f}_i=0,$$
where the last equation is in the vector space of complex polynomials of degree less or equal to $\prod_{i=1}^N d_i$.

\begin{exemple}
Consider polynomial of degree $4$ over $\mathbb{F}_2$, prime to $x+1$, we obtain an element of the kernel in degree $12$ : 
$$P_{FGH}-5P_{EGH}+3P_{EFH}+P_{EFG},$$
where
$$\begin{array}{ccc}
E &= & x^4+x^3+x^2+x+1\\
F &= & x^4+x^3+1\\
G &= & x^4+x^2 +1=(x^2+x+1)^2\\
H & = & x^4+x+1.
\end{array}$$
By definition, this is a $1$-eigenvector for the action of $T$, and this gives an answer to \cite[Conj. 5.6]{Hai19}.
\end{exemple}
\section{Acknowledgements}
The author thanks {\fontencoding{T5}\selectfont Nguyễn  Đặng Hồ Hải} for sharing ideas during his visit at Jean Leray Laboratory in the Spring of 2018. This visit was supported by the regional project Defimaths-Pays de la Loire. She also thanks the Centre Henri Lebesgue ANR-11-LABX-0020-01 and the ANR-16-CE40-0003 ChroK project for their support.
\bibliographystyle{plain}
\bibliography{vpT}
\end{document}